\documentclass{amsart}

\newtheorem{theorem}{Theorem}[section]
\newtheorem{lemma}[theorem]{Lemma}
\newtheorem{corollary}{Corollary}[section]

\theoremstyle{definition}
\newtheorem{remark}{Remark}[section]

\numberwithin{equation}{section}

\usepackage[nobysame,initials]{amsrefs}
\usepackage{hyperref}
\usepackage{xcolor}

\allowdisplaybreaks

\newcommand{\R}{\mathbb{R}}
\newcommand{\EE}{\mathbb{E}}
\newcommand{\NN}{\mathbb{N}}

\newcommand{\cl}{{\rm cl }\,}

\newcommand{\dx}{{\rm d}}

\newcommand{\bd}{\mathrm{bd}\,}

\title{On random approximations by generalized disc-polygons}

\author{Ferenc Fodor$^1$} 
\address{Department of Geometry, Bolyai Institute, University of Szeged, 
Aradi v\'ertan\'uk tere 1, 6720 Szeged, Hungary} 
\email{fodorf@math.u-szeged.hu}

\author{D\'aniel I. Papv\'ari$^2$}
\address{University of Szeged, Aradi v\'ertan\'uk tere 1, 6720 Szeged, Hungary}
\email{papvari.daniel.istvan@stud.u-szeged.hu}

\author{Viktor V\'{\i}gh$^3$}
\address{Department of Geometry, Bolyai Institute, 
University of Szeged, Aradi v\'ertan\'uk tere 1, 6720 Szeged, Hungary}
\email{vigvik@math.u-szeged.hu}

\thanks{$^{1,2,3}$ Supported by the Ministry of Human Capacities, Hungary grant 20391-3/2018/FEKUSTRAT}
\thanks{$^{1,3}$ Supported by Hungarian National 
	Research, Development and Innovation Office NKFIH grant K 116451.}

\subjclass[2010]{Primary 52A22, Secondary 52A27, 60D05}

\begin{document}
\bibliographystyle{amsplain}

\begin{abstract}
For two convex discs $K$ and $L$, we say that $K$ is $L$-convex \cite{L} if it is equal to the intersection of all translates of $L$ that contain $K$. In $L$-convexity the set $L$ plays a similar role as closed half-spaces do in the classical notion of convexity.  We study the following probability model: Let $K$ and $L$ be $C^2_+$ smooth convex discs such that $K$ is $L$-convex. Select $n$ i.i.d. uniform random points $x_1,\ldots, x_n$ from $K$, and consider the intersection $K_{(n)}$ of all translates of $L$ that contain all of $x_1,\ldots, x_n$. The set $K_{(n)}$ is a random $L$-convex polygon in $K$. We study the expectation of the number of vertices $f_0(K_{(n)})$ and the missed area $A(K\setminus K_{n})$ as $n$ tends to infinity. We consider two special cases of the model. In the first case we assume that the maximum of the curvature of the boundary of $L$ is strictly less than $1$ and the minimum of the curvature of $K$ is larger than $1$. In this setting the expected number of vertices and missed area behave in a similar way as in the classical convex case and in the $r$-spindle convex case (when $L$ is a radius $r$ circular disc), see \cite{FKV14}. 
The other case we study is when $K=L$. This setting is special in the sense that an interesting phenomenon occurs: the expected number of vertices tends to a finite limit depending only on $L$. This was previously observed in the special case when $L$ is a circle of radius $r$ in \cite{FKV14}. We also determine the extrema of the limit of the expectation of the number of vertices of $L_{(n)}$ if $L$ is a convex discs of constant width $1$. The formulas we prove can be considered as generalizations of the corresponding $r$-spindle convex statements proved by Fodor, Kevei and V\'\i gh in \cite{FKV14}. 
\end{abstract}

\maketitle

\section{Introduction and results}
R\'enyi and Sulanke started the investigation of the asymptotic properties of random polytopes in their seminal papers \cites{RS63, RS64, RS68}. They studied the planar version of the following probability model: Let $K$ be a convex body (compact convex set with interior points) in Euclidean $d$-space $\R^d$, and select $n$ i.i.d random points $x_1, \ldots, x_n$ from $K$ according to the uniform probability distribution. The convex hull of the random points $x_1, \ldots, x_n$ is a (random) polytope $K_n$ in $K$, which tends to $K$ with probability $1$ as $n\to\infty$. Common random variables associated with such polytopes are, for example, the number of $i$-dimensional faces for $i=0,\ldots, d-1$, and the difference of the $j$th intrinsic volumes of $K$ and $K_n$ for $j=1,\ldots, d-1$. After the works of R\'enyi and Sulanke, many of the results in the theory of random polytopes have been of asymptotic type, meaning that they describe the limiting behaviour of some aspect, such as expectation, variance, etc., of a random variable as the number of points $n$ tends to infinity. Our motivations come, in part, from the asymptotic formulas proved by R\'enyi and Sulanke for the expected number of vertices (\cite[Satz 3. p. 83]{RS63}) and the missed area (\cite[Satz 1. (48) p. 144]{RS64}) of random convex polygons in sufficiently smooth convex discs. Our aim is to prove similar statements in a different, and somewhat more general, setting in the Euclidean plane. In the last few decades the literature on this topic has grown enormously, especially in the general $d$-dimensional setting. We do not venture to give an overview of the subject in this paper, instead, we refer to the comprehensive surveys \cites{Bar08, H13, R10, Schneider, Sch18, Schn8, Sch08, Weil93} for more information and references.

Recently, another probability model of random polytopes emerged that is based on intersections of congruent closed balls of suitable radius, see \cites{F19, FKV14, FV18}. For a fixed $r>0$, a convex disc $K\subset\R^2$ is called $r$-spindle convex (sometimes also called $r$-hyperconvex \cite{FTGF15} or $r$-convex \cites{FT182, FT282}) if, together with any two points $x,y\in K$, the set $[x,y]_r$, consisting of all shorter circular arcs of radius at least $r$ and connecting $x$ and $y$, is contained in $K$. One can also think of $[x,y]_r$ as the intersection of all radius $r$ closed circular discs that contain $x$ and $y$. In this concept, the set $[x,y]_r$ plays a similar role as the segment in the classical notion of convexity. The intersection of a finite number of radius $r$ circles is called a disc polygon of radius $r$. The concept of spindle convexity emerged from a paper of Mayer \cite{M35}, and has subsequently been investigated from different points of view. For more information on spindle convex sets and further references we refer to the paper by Bezdek, L\'angi, Nasz\'odi and Papez \cite{BL07} and the recent book by Martini, Montejano and Oliveros \cite{MMO19}. We only note that the importance of spindle convexity lies, at least partly, in the role intersections of congruent balls play in the study of, for example, the Kneser-Poulsen conjecture, diametrically complete bodies, randomized isoperimetric inequalities, etc., for more on this topic and references we suggest to consult \cites{BL07, FV12, FTGF15, FoKuVi16, MMO19, PP17}.

If one selects $n$ i.i.d. random points $x_1,\ldots, x_n$ from an $r$-spindle convex disc $K$ according to the uniform probability distribution, then the intersection of all radius $r$ discs that contain $x_1,\ldots, x_n$ is a random disc polygon $K_{(n)}^r$ of radius $r$ in $K$. Due to the $r$-spindle convexity, this random disc polygon is contained in $K$. In a recent paper, Fodor, Kevei and V\'igh \cite{FKV14} proved asymptotic formulas for the expectation of the number of vertices, missed area, and perimeter difference of $K_{(n)}^r$ under suitable smoothness assumption on the boundary of $K$. 
These asymptotic formulas are generalizations of the corresponding classical results of R\'enyi and Sulanke in the limit as $r\to\infty$. 
Asymptotic estimates on the variance of the number of vertices and missed area were established in \cite{FV18} for smooth $r$-spindle convex disc.
The $r$-spindle convex probability model was generalized to $d$ dimensions in \cite{F19} where an asymptotic formula was proved for the expected number of proper facets of the resulting random ball-polytope (\cite{F19}*{Theorem~1.1}) in the case when a ball of radius $r$ is approximated by random ball-polytopes of radius $r$, and asymptotic upper and lower bounds were established for the expected number of proper facets for general convex bodies with sufficiently smooth boundary and suitable radius $r$.  

The notion of spindle convexity can further be generalized by replacing the radius $r$ circular disc by a fixed convex disc $L$. This leads to the notions of $L$-convexity and $L$-spindle convexity, as introduced in \cite{L}. For a historical overview of this topic and references consult the Introduction of \cite{L}.

Let $K$ and $L$ be convex discs (to avoid technical complications we always assume that the sets involved are compact). We say that $K$ is $L$-convex (\cite{L}*{Definition~1.1}) if it is equal to the intersection of all translates of $L$ that contain $K$. Of course, if $K$ is $L$-convex, then it is also convex in the usual sense. 
Let $X\subset\R^2$ be a set contained in a translate of $L$. We denote the intersection of all translates of $L$ that contain $X$ by $[X]_L$. The set $[X]_L$ is called the $L$-convex hull of $X$. If $L$ is strictly convex and $X$ has at least two points, then the interior of $[X]_L$ is non-empty.

We say that the convex disc $K$ is $L$-spindle convex (\cite{L}*{Definition~1.2}) if it is contained in a translate of $L$ and for any $x,y\in K$ it holds that $[x,y]_L\subset K$. 
It is clear that if $K$ is $L$-convex, then it is also $L$-spindle convex. The converse is also true (in the plane), see \cite[Corollary~3.13, p. 51]{L}. Thus, in our case the two notions of convexity determined by $L$ are equivalent and can be used interchangeably. We note that $L$-convexity and $L$-spindle convexity can be defined analogously in $d$ dimension as well, but from $d\geq 3$ the two properties are no longer equivalent, see \cite[Theorem~3, p. 48]{L}. 

In classical convexity a closed convex set is known to have a supporting hyperplane through any of its boundary points. We say that a  convex set is smooth if this supporting hyperplane is unique at each boundary point. A similar property holds for $L$-convex discs too, see \cite[Theorem~4, p. 50]{L}: If $K$ is $L$-convex, $x\in \bd K$ and $l$ is a supporting line of $K$ through $x$, then there exists a translate $L+p$  such that $x\in L+p$, $K\subset L+p$ and  $l$ supports $L+p$ at $x$. In this case we call $L+p$ a supporting disc of $K$ at $x$.  
It clearly follows that if both $K$ and $L$ are smooth, then $K$ has a unique supporting disc at each boundary point. 

We note that the existence of a supporting translate of $L$ at each point of $\bd K$ is also known as the property that $K$ slides freely in $L$, see \cite[p. 156]{Schneider}.

The $L$-convex property is invariant under translations of $K$, and also under homotheties with ratio strictly between $0$ and $1$, that is, $K$ is $L$-convex if and only if $\lambda K+p$ is $L$-convex for all $\lambda\in (0,1)$ and $p\in \R^2$, see \cite[Corollary~3.7, p. 47]{L}).

We study the following probability model. Let $K$ and $L$ be convex discs with $C^2_+$ smooth boundary (twice continuously differentiable with strictly positive curvature everywhere) such that $K$ is $L$-convex. The $C^2_+$ property yields that both $K$ and $L$ are strictly convex. 
 Let $x_1,\ldots, x_n$
be i.i.d. random points from $K$ selected according to the uniform probability distribution. We call $K_{(n)}=[x_1,\ldots, x_n]_L$ a uniform random $L$-polygon contained in $K$. A point $x_{i_j}\in \{x_1,\ldots, x_n\}$ is a vertex of $K_{(n)}$ if it is a non-smooth point of $\bd K_{(n)}$. The vertices $x_{i_1}, \ldots x_{i_k}$, $2\leq k\leq n$ of $K_{(n)}$ divide $\bd K_{(n)}$ into $k$ arcs, which we call sides, each of which is a connected arc of the boundary of a translate of $L$.  Let $f_0(K_{(n})$ denote the number of vertices of $K_{(n)}$, and $A(K\setminus K_{(n)})$ the missed area. We investigate the asymptotic behaviour of the expectations of $f_0(K_{(n})$ and $A(K\setminus K_{(n)})$. 

Our paper contains the discussion of two special cases of this probability model. In the firs one, we  make the following further assumption on the curvatures of the boundaries of $K$ and $L$:
\begin{equation}\label{cuvr-cond}
\max_{x\in\bd L} \kappa_L (x)<1< \min_{y\in\bd K} \kappa_K(y),
\end{equation}
where $\kappa_L(x)$ is the curvature of $\bd L$ at $x$ and $\kappa_K(y)$ is the curvature of $\bd K$ at $y$. 

We note 
that \cite[Theorem~3.2.12, p. 164]{Schneider} states that for two $C^2_+$ smooth convex discs, $K$ slides freely in $L$ if and only if the curvature of $\bd K$ is at least as large as the curvature of $\bd L$ in points where the outer unit normals are equal. This condition is clearly satisfied under the  assumption \eqref{cuvr-cond}, thus, in this case $K$ slides freely in $L$ and so $K$ is $L$-convex.

Under the assumption \eqref{cuvr-cond}, the expected number of vertices and the missed area both behave in a similar manner as in the usual convex case.

\begin{theorem}\label{fotetel}
	With the above assumptions
\begin{equation}
\lim_{n\to\infty}\EE (f_0(K_{(n)}))n^{-\frac 13}=\sqrt[3]{\frac{2}{3 A(K)}}\Gamma\left(\frac53\right)\int_{S^1}\frac{\left(\kappa_K(u)-\kappa_L(u)\right)^{\frac 13}}{\kappa_K(u)}\dx u.
\end{equation}
\end{theorem}

Efron's identity \cite{E65}, which, in two dimensions, relates the expectation of the number of vertices and the missed area, can be easily extended to the  $L$-convex probability model as follows
$$\EE (f_0(K_{(n)}))=n \frac{1}{A(K)}\EE (A(K\setminus K_{(n-1)})).$$
Thus we obtain the following corollary of Theorem~\ref{fotetel}:

\begin{corollary}\label{fotetel:terulet}
	With the same conditions as above
\begin{equation}
\lim_{n\to\infty}\EE (A(K\setminus K_{(n)}))n^{\frac 23}=\sqrt[3]{\frac{2A^2(K)}{3}}\Gamma\left(\frac53\right)\int_{S^1}\frac{\left(\kappa_K(u)-\kappa_L(u)\right)^{\frac 13}}{\kappa_K(u)}\dx u.
\end{equation}
\end{corollary}

We note that Theorem~\ref{fotetel} and Corollary~\ref{fotetel:terulet} are contained in D. Papv\'ari's Bachelor's thesis \cite{P19}. The proof of Theorem~\ref{fotetel} is also from \cite{P19}. 

Notice that in both Theorem~\ref{fotetel} and Corollary~\ref{fotetel:terulet}, we get back the corresponding statements of Fodor, Kevei and V\'\i gh \cite[Theorem 1.1]{FKV14} when $L$ is a circle of radius $r>1$. 

In the other special case of the probability model we investigate in this paper we assume that $K=L$, and denote the corresponding random $L$-convex polygon by $L_{(n)}$. This leads to an interesting phenomenon that cannot be observed in the usual convex case. Namely, the expectation of the number of vertices tends to a finite limit determined by only $L$. This has already been pointed out in the case when $L=B^2$ in the paper by Fodor, Kevei and V\'\i gh, see \cite{FKV14}*{Theorem~1.3}, and in $d$ dimensions for $L=B^d$ by Fodor \cite{F19}*{Theorem~1.1}.

To formulate a precise statement we introduce the following notation. For a unit vector $u\in S^1$, we denote by $w_L(u)=w(u)$ the distance between the two supporting lines $l_1(u)$ and $l_2(u)$ of $L$ that are parallel to $u$. This is the well-known width of $L$ in the direction $u^\perp$ orthogonal to $u$.

\begin{theorem}\label{LL}
Let $L$ be a convex disc with $C^2_+$ boundary. Then
\begin{align}
 \lim_{n\to\infty}\EE (f_0(L_{(n)}))=\pi\int_{S^1} \frac{1}{\kappa^2_L(u)\cdot w^2(u)} \dx u, \label{vertex-number-L}\\
 \lim_{n\to\infty} \EE (A(L\setminus L_{(n)}))\cdot n=A(L)\pi\int_{S^1} \frac{1}{\kappa^2_L(u)\cdot w^2(u)} \dx u.\label{missed-area-L}
\end{align}
\end{theorem}

We remark that \eqref{missed-area-L} follows from \eqref{vertex-number-L} by Efron's identity, thus we focus only on the number of vertices. We note that Theorem~\ref{LL} is particularly interesting in the case when $L$ is a convex disc of constant width $1$. (The expected number of vertices is clearly scaling invariant.) It is well-known that if $L$  has constant width $1$, then $\kappa_L^{-1}(u)+\kappa_L^{-1}(-u)=1$ (see, for example, in a more general setting on p. 341 in \cite{FoKuVi16}). This implies $\kappa_L^{-1}(u)<1$, and thus
$$\lim_{n\to\infty}\EE (f_0(L_{(n)}))=\pi\int_{S^1} \frac{1}{\kappa^2_L(u)\cdot w^2(u)} \dx u=\pi\int_{S^1} \frac{1}{\kappa^2_L(u)} \dx u<2\pi^2. $$
Also, by the arithmetic mean/quadratic mean inequality
$$\frac 14=\left (\frac{\kappa_L^{-1}(u)+\kappa_L^{-1}(-u)}{2} \right )^2 \le \frac{\kappa_L^{-2}(u)+\kappa_L^{-2}(-u)}{2},$$
 which implies
$$\lim_{n\to\infty}\EE (f_0(L_{(n)}))=\pi\int_{S^1} \frac{1}{\kappa^2_L(u)\cdot w^2(u)} \dx u=\pi\int_{S^1} \frac{1}{\kappa^2_L(u)} \dx u\ge\frac{\pi^2}{2}. $$
We note that both inequalities are sharp. The upper bound can be approximated by smoothed Reuleaux-polygons. Reuleaux-polygons are not smooth, however with a slight modification at the vertices one can construct a smooth convex disc of constant width $1$ such that the limit of the expectation of the number of the vertices is arbitrarily close to $2\pi^2$. The lower bound is achieved when $L$ is a circle, as it was shown in Theorem 1.3 in \cite{FKV14}.

We also note that if $L$ is a convex disc with $C^2_+$ boundary (but not necessarily of constant width), then the limit is still clearly bounded from below by $2$, but one can construct a sausage-like domain with arbitrarily large limit.

\section{Caps of $L$-convex discs}
In this section we assume that \eqref{cuvr-cond} holds for $K$ and $L$.
We call a subset $C$ of $K$ an {\em $L$-cap} if $C=\cl (K\setminus (L+p))$ for some $p\in \R^2$. Here $\cl(\cdot)$ denotes the closure of a set. Due to the condition \eqref{cuvr-cond} on the curvatures of the boundaries of $K$ and $L$, the curves $\bd K$ and $\bd L+p$ have exactly two intersection points. These two intersection points divide $\bd C$ into two parts, one belongs to $\bd K$ and the other one to $\bd L+p$. Below we state three technical lemmas that will be used in the subsequent arguments. We note that these lemmas are the $L$-convex analogues of the corresponding $r$-spindle convex statements in \cite{FKV14}, see Lemmas~4.1--4.2. 

For a smooth convex disc $M$, the unique outer unit normal at $x\in\bd M$ is denoted by
$u(M,x)$. If $M$ is also strictly convex, then for each $u\in S^1$ there exists a unique point $x=x(M,u)$ such that the outer unit normal of $\bd M$ at $x$ is $u$, that is, the functions $x(M,u)$ and $u(M,x)$ are inverse to each other. If $\bd M$ is $C^2_+$, then, with a slight abuse of notation, we use $\kappa_M(u)=\kappa_M(x(M,u))$ for the curvature of $\bd M$.

\begin{lemma}\label{capvertexexistsandunique}
Let $K$ and $L$ be as above. For an $L$-cap $C=\cl (K\setminus (L+p))$, there exists a unique point $x_0\in\bd C\cap \bd K$ and $t\geq 0$ such that
$y_0=x_0-t u(K,x_0)\in\bd D\cap (\bd L+p)$ and $u(L+p,y_0)=u(K,x_0)$. 
\end{lemma}

\begin{proof}
We may assume, without loss of generality, that $p=0$.
The existence of $x_0$ and $t$ follows from the following standard continuity argument. Let $y_1$ and $y_2$ be the two intersection points of $\bd K$ and $\bd L$ in the positive direction on $\bd L$. For a point $y\in \bd C\cap\bd L$, there exists a unique $t\geq 0$ such that $x=y+tu(L,y)\in\bd C\cap\bd K$, that is, $x$ is the intersection point of the ray with end-point $y$, direction $u(L,y)$ and $\bd K$. It is clear that $x$ is strictly monotonically increasing with $y$.	Let $\varphi =\varphi (y)$ be the signed angle $u(L,y)$ and $u(K,x)$. Then $\varphi(y_1)<0$ and $\varphi(y_2)>0$, and $\varphi$ is a continuous, in fact, continuously differentiable function of $y$. Therefore, there is a $y_0$ such that $\varphi (y_0)=0$. Thus, $y_0$ and the corresponding $x_0$ and $t_0$ satisfy the statement of the lemma. 

Next, we prove the uniqueness of $y_0$. On the contrary, assume that there is another point, say $y_0'$ with the same property.

Let $x_0$ and $x_0'$ be the point on $\bd C\cap\bd K$ corresponding to $y_0$ and $y_0'$. First, note that 
\begin{equation}\label{chord}
d(y_0,y_0')<d(x_0,x_0').
\end{equation}
Clearly, due to the $C_+^2$ property of $\bd K$ and $\bd L$, the lines $x_0y_0$ and $x_0'y_0'$ intersect in a point, say $p$.
Let $\psi$ denote the angle of $u(L,y_0)$ and $u(L,y_0')$ (which is the same as the angle of $u(K,x_0)$ and $u(K,x_0')$). 

Due to the relative position of $K$ and $L$, $d(x_0,p)>d(y_0,p)$ and $d(x'_0,p)>d(y'_0,p)$, and all angles of the triangles $py_0y_0'$ and $px_0x_0'$ are non-obtuse (since the perpendiculars at $x_0,x_0'$, and at $y_0,y_0'$, are supporting lines of $K$, and  $L$, respectively). If $\angle py_0'y_0<\angle px_0'x_0$, then let $y_0''$ be the intersection point of $py_0$ and the line parallel to $y_0y_0'$ through $x_0'$. Then,  since $\angle py_0y_0'>\angle px_0x_0'$, it holds that $d(y_0',y_0)<d(x_0',y_0'')<d(x_0',x_0)$, proving \eqref{chord} in this case. The case when $\angle py_0'y_0>\angle px_0'x_0$ is similar.

Again, by the conditions on the curvatures of $\bd K$ and $\bd L$, the shorter open arc of the unit circle with end-points $y_0$ and $y_0'$ is in $K$ but outside of $L$. 
Therefore, its length $h$ is larger than the length $\Delta s$ of the arc of $\bd L$ from $y_0$ to $y_0'$. Similarly, the shorter closed unit circular arc with end-points $x_0$ and $x_0'$ is completely in $K$, and thus its length $h'$ is less than the length $\Delta s'$ of the arc of $\bd K$ from $x_0$ to $x_0'$. 
It follows from \eqref{chord} that $h<h'$.
In summary,
$$\Delta s< h<h'<\Delta s'.$$
On the other hand, if $I_K$ denotes the part of $\bd K$ between $x_0$ and $x_0'$, and $I_L$ denotes the part of  $\bd L$ between $y_0$ and $y_0'$, then it follows from the conditions on the curvatures of $\bd L$ and $\bd K$ that
	$$\Delta s=\int_{I_L}\dx s>\int_{I_L}\kappa_L(s)\dx s=\psi=\int_{I_K}\kappa_K(s')\dx s'>\int_{I_K}\dx s'=\Delta s',$$
which is a contradiction.
\end{proof}
This (unique) point $x_0=x(K,u)$ is usually called the {\em vertex} of $C$ and the corresponding $t$ is the {\em height}. Since the $L$-cap $C$ is uniquely determined by its vertex and height, we introduce the notation $C(u,t)$ to denote such a cap. This, in fact, provides a parametrization of $L$-caps in terms of a unit vector and a (sufficiently small) positive real number. Let $A(u,t)=A(C(u,t))$, and let $\ell(u,t)$ be the arc-length of $C\cap(\bd L+p)$.

\begin{lemma}\label{landAasym}
Let $K$ and $L$ be as above. Then, for a fixed $u\in S^1$, the following hold:
	\begin{align}
	\lim_{t\to 0^+}\ell(u,t)\cdot t^{-\frac 12}=&2\cdot\sqrt{\frac{2}{\kappa_K(u)-\kappa_L(u)}},\label{L-arc}\\
	\lim_{t\to 0^+}A(u,t)\cdot t^{-\frac 32}=&\frac 43\cdot\sqrt{\frac{2}{\kappa_K(u)-\kappa_L(u)}}.\label{eq:Aut}
	\end{align}
\end{lemma}
The proof of Lemma~\ref{landAasym} is very similar to that of Lemma~4.2 in \cite[p. 906]{FKV14}, thus we omit the details. The main idea of the argument is that we assume that $x=(0,0)$ and $u=(0,-1)$. Then, in a sufficiently small open neighbourhood of the origin, $\bd K$ is the graph of a $C^2$ smooth convex function $f(x)$. Then we use the second order Taylor expansion of $f$ around the origin from which we obtain the statements of the lemma by simple integration.

For two points, $x,y\in K$ there are the two (unique) translates of $L$ such that each translate contains both $x$ and $y$ on its boundary.  We denote the $L$-caps determined by these translates of $L$ by $C_-(x,y)$ and $C_+(x,y)$ with the assumption that $A_-(x,y)=A(C_-(x,y))\leq A(C_+(x,y))=A_+(x,y)$.

\begin{lemma}\label{lemma:delta}
Let $K$ and $L$ be as above. Then there exists a constant $\delta>0$ such that for any $x_1, x_2\in K$, the $A_+(x_1,x_2)>\delta$. The constant $\delta$ depends only on $K$ and $L$. 
\end{lemma}
We also omit the proof of Lemma~\ref{lemma:delta} as it is essentially the same as that of Lemma~4.3 in \cite[p. 906]{FKV14}; it uses only the conditions on the curvatures of $\bd L$ and $\bd K$ and a simple compactness argument.

Finally, we need the existence of a rolling circle in $K$. We say that a circle of radius $\varrho>0$ rolls freely in $K$ if each $x\in\bd K$ is in a closed circular disc of radius $\varrho$ that is fully contained in $K$. It follows from Blaschke's result \cite{B56} that if $\bd K$ is  
$C^2$ smooth with the above conditions on its curvature, then there exists a circle of radius $0<\varrho<1$ (cf. \cite[p. 906]{FKV14} and \cite{Hug99}) that rolls freely in $K$. Thus, according to \eqref{eq:Aut} there exists a $0<t^*<\varrho$, such that for all $u\in S^1$
\begin{equation}\label{eq:gordulo}
A(u,t)\geq\frac 12\left(\frac 43 \sqrt{\frac{2}{1/\varrho-\kappa^*}}\right)t^{\frac 32},\quad \text{if }t\in[0,t^*],
\end{equation}
where $\kappa^*=\min_{u\in S^1} \kappa_L(u)$.

\section{Proof of the Theorem~\ref{fotetel}}\label{pf}
Our argument is essentially based on ideas that originated from R\'enyi and Sulanke \cite{RS63}, and which were also used in the spindle convex setting in \cite{FKV14}. Here we generalize and apply them to the $L$-convex probability model. 

A pair of random points $x_i, x_j$ forms an edge of $K_{(n)}$ if at least one of the $L$-caps $C_-(x_i, x_j)$ and $C_+(x_i, x_j)$ contains no other points of $x_1, \ldots, x_n$. Let $A_-(x,y)=A(C_-(x,y))$ and $A_+(x,y)=A(C_+(x,y))$. Then
\begin{multline}\label{eq:expectation}
\EE \big(f_0(K_{(n)})\big)=\frac{1}{A(K)^2}{n\choose 2}\int_{K}\int_K\left [\left (1-\frac{A_-(x_1, x_2)}{A(K)}\right )^{n-2}\right.\\
\left.+\left (1-\frac{A_+(x_1, x_2)}{A(K)}\right )^{n-2}\right ] \dx x_1\dx x_2,
\end{multline}
where integration is with respect to the Lebesgue measure in $\R^2$. 

Using a similar argument to the one in \cite[p. 907]{FKV14} one can show that the contribution of the second term of \eqref{eq:expectation} in the limit as $n\to\infty$ is negligible, in fact, it is exponentially small. For the sake of completeness, we give a detailed proof. For any fixed $\alpha\in \R$, it follows from Lemma~\ref{lemma:delta} that
\begin{align*}
\lim_{n\to\infty} &n^{\alpha}\frac{1}{A(K)^2}{n\choose 2}\int_{K}\int_K\left (1-\frac{A_+(x_1, x_2)}{A(K)}\right )^{n-2} \dx x_1\dx x_2\\
&\leq \lim_{n\to\infty} n^{\alpha}\frac{1}{A(K)^2}{n\choose 2}\int_{K}\int_K \left (1-\frac{\delta}{A(K)}\right )^{n-2}\dx x_1\dx x_2\\
&\leq \lim_{n\to\infty} n^{\alpha}\frac{1}{A(K)^2}{n\choose 2}\int_{K}\int_K e^{-\frac{\delta (n-2)}{A(K)}}\dx x_1\dx x_2\\
&=\lim_{n\to\infty} n^{\alpha}{n\choose 2}e^{-\frac{\delta (n-2)}{A(K)}}\\
&=0.
\end{align*}

Note that the same argument shows that the contribution of those pairs $x_1, x_2$ for which $A_-(x_1, x_2)>\delta$ is also negligible in the limit. Thus,
\begin{multline}\label{forintegral}
\lim_{n\to\infty} \EE \big(f_0(K_{(n)})\big)n^{-\frac 13}=\\
\lim_{n\to\infty} n^{-\frac 13}\frac{1}{A(K)^2}{n\choose 2}\int_{K}\int_K\left (1-\frac{A_-(x_1, x_2)}{A(K)}\right )^{n-2}{\bf 1}(A_-(x_1,x_2)<\delta) \dx x_1\dx x_2,
\end{multline}
where ${\bf 1}(\cdot)$ denotes the indicator function of an event.
In the rest of the proof we evaluate the right hand side of \eqref{forintegral}.

Let 
\begin{equation}\label{forjacobi}
    (x_1, x_2)=\Phi (u,t, u_1, u_2),
\end{equation}
where $u\in S^1$ and $t\leq t_0$ are such that $C(u,t)=C_-(x_1, x_2)$. Let $L(u,t)$ denote the arc $C(u,t)\cap(\bd L+x(K,u)-x(L,u)-tu)$. Thus $x_1, x_2\in L(u,t)$. 
The outer unit normals of $L+x(K,u)-x(L,u)-tu$ on the arc $L(u,t)$ determine a connected arc of $S^1$, which we denote by $L^*(u,t)$. 
Let $u_1, u_2$ be the outer unit normals of $L+x(K,u)-x(L,u)-tu$ at $x_1$ and $x_2$. Thus,
\begin{equation}\label{notation}
   x_i=x(K,u)-x(L,u)-tu+x(L,u_i), \quad i=1,2, 
\end{equation}
with $u_1, u_2\in L^*(u,t)$.

Lemma~\ref{capvertexexistsandunique} guarantees the uniqueness of the vertex and height of an $L$-cap, thus $\Phi$ is well-defined, bijective, and differentiable (see the Appendix) on a suitable domain of $(u,t,u_1,u_2)$ with the possible exception of a set of measure zero. The Jacobian of the transformation $\Phi$ is
\begin{equation}\label{jacobi}
  \left|J\Phi\right|=\frac{\left|u_1\times u_2\right|}{\kappa_L(u_1)\kappa_L(u_2)}\left(\frac{1}{\kappa_L(u)}-\frac{1}{\kappa_K(u)}+t\right),  
\end{equation}
see the details in the Appendix.
From \eqref{forintegral} and \eqref{jacobi}, we obtain
\begin{multline}\label{int:1}
\lim_{n\to\infty}\EE (f_0(K_{(n)}))n^{-\frac 13}\\=
\lim_{n\to\infty} n^{-\frac 13}\frac{1}{A(K)^2}{n\choose 2}\int_{S^1}\int_{0}^{t^*(u)}\int_{L^*(u,t)}\int_{L^*(u,t)}\left (1-\frac{A(u,t)}{A(K)}\right )^{n-2}\\
\times \frac{\left|u_1\times u_2\right|}{\kappa_L(u_1)\kappa_L(u_2)}\left(\frac{1}{\kappa_L(u)}-\frac{1}{\kappa_K(u)}+t\right) \dx u_1\dx u_2\dx t\dx u,
\end{multline}
with a suitable $t^*(u)$ depending only on $K$ and $L$. 

We note that in \eqref{int:1} we can replace $t^*(u)$ by any fixed $0<t_1\leq t^*(u)$ and the limit remains unchanged. We choose a suitable $0<t_1\leq t^*(u)$ such that $A(u,t)\geq \delta$ for all $t_1\leq t\leq t^*(u)$ and all $u\in S^1$. 

Now we split the domain of integration with respect to $t$ into two parts. Let $h(n)=(c\ln n/n)^{2/3}$, where $c$ is a suitable positive constant specified below in the proof. There exists $n_0\in \NN$, such that if $n>n_0$ then $h(n)<t_1$. Furthermore, there also exists $\gamma_1>0$ constant such that $A(u,t)>\gamma_1 h(n)^{3/2}$ for all $u\in S^1$ and $h(n)< t\leq t_1$. For $u\in S^1$ and $0\leq t\leq t_1$, let
	$$I^*(u,t)=\int_{L^*(u,t)}\int_{L^*(u,t)}
	\frac{\left|u_1\times u_2\right|}{\kappa_L(u_1)\kappa_L(u_2)} \dx u_1\dx u_2$$
	and
	$$k(u,t)=\frac{1}{\kappa_L(u)}-\frac{1}{\kappa_K(u)}+t.$$
\begin{lemma}\label{lemma:h(n)} Let $h(n)$ be defined as above.
	Then
	$$
	\lim_{n\to\infty} n^{-\frac 13}\frac{1}{A(K)^2}{n\choose 2}\int_{S^1}\int_{h(n)}^{t_1}\left (1-\frac{A(u,t)}{A(K)}\right )^{n-2}k(u,t)
	%\\\times\left(\frac{1}{\kappa_L(u)}-\frac{1}{\kappa_K(u)}+t\right)\cdot 
	I^*(u,t)\dx t\dx u=0.
	$$
\end{lemma}

\begin{proof}Note that there exists a universal constant $\gamma_2>0$ such that
	$$%\left(\frac{1}{\kappa_L(u)}-\frac{1}{\kappa_K(u)}+t\right)
	k(u,t) I^*(u,t)\leq \gamma_2$$
	for all $u\in S^1$ and $0< t\leq t_1$. Hence, for a fixed $u\in S^1$ and any $n>n_0$, it holds
	\begin{align*}
	\int_{h(n)}^{t_1}&\left (1-\frac{A(u,t)}{A(K)}\right )^{n-2}
	%\left(\frac{1}{\kappa_L(u)}-\frac{1}{\kappa_K(u)}+t\right)\cdot 
	k(u,t) I^*(u,t)\dx t\\
	&\ll \int_{h(n)}^{t_1}\left (1-\frac{A(u,t)}{A(K)}\right )^{n-2}\dx t\\
	&\ll \int_{h(n)}^{t_1}\left (1-\frac{\gamma_1h(n)^{\frac 32}}{A(K)}\right )^{n-2}\dx t\\
	&\ll \int_{0}^{t_1}\left (1-\frac{\gamma_1 c (\ln n/n)}{A(K)}\right )^{n-2}\dx t\\ 
	&\ll n^{-\frac{\gamma_1 c}{A(K)}}.
	\end{align*}
	If $c>5A(K)/(3\gamma_1)$, then
	\begin{align*}
	n^{-\frac 13}&\frac{1}{A(K)^2}{n\choose 2}\int_{S^1}\int_{h(n)}^{t_1}
	\left (1-\frac{A(u,t)}{A(K)}\right )^{n-2}
	%\left(\frac{1}{\kappa_L(u)}-\frac{1}{\kappa_K(u)}+t\right)\cdot 
	k(u,t)I^*(u,t)\dx t\dx u
	&\ll n^{\frac 53}n^{-\frac{\gamma_1 c}{A(K)}},
	\end{align*}
	which clearly converges to $0$ as $n\to\infty$.
\end{proof}

Let $\varepsilon>0$ be fixed. There exists a $0<t_\varepsilon<t_1$ such that for all $0<t<t_\varepsilon$, $u\in S^1$ and for any $u_1, u_2\in L^*(u,t)$
\begin{equation*}
(1-\varepsilon)\frac{1}{\kappa^2_L(u)}<\frac{1}{\kappa_L(u_1)\kappa_L(u_2)}< (1+\varepsilon)\frac{1}{\kappa^2_L(u)}.
\end{equation*}
Then, with the notation 
$$I(u,t)=\int_{L^*(u,t)}\int_{L^*(u,t)}
\left|u_1\times u_2\right| \dx u_1\dx u_2,$$
we obtain
$$\frac{1-\varepsilon}{\kappa^2_L(u)}I(u,t)
%\int_{L^*(u,t)}\int_{L^*(u,t)}\!\! |u_1\times u_2| \dx u_1\dx u_2
<I^*(u,t)
<\frac{1+\varepsilon}{\kappa^2_L(u)}I(u,t)
%\int_{L^*(u,t)}\int_{L^*(u,t)}\!\!|u_1\times u_2| \dx u_1\dx u_2
.$$
And
\begin{equation*}
%\int_{L^*(u,t)}\int_{L^*(u,t)}\left|u_1\times u_2\right| \dx u_1\dxu_2
I(u,t)=2(\ell^*(u,t)-\sin\ell^* (u,t)),
\end{equation*}
where $\ell^*(u,t)$ is the length of the arc $L^*(u,t)\subset S^1$. Thus
\begin{equation}\label{kell}
I^*(u,t)=
%\int_{L^*(u,t)}\int_{L^*(u,t)}
%\frac{\left|u_1\times u_2\right|}{\kappa_L(u_1)\kappa_L(u_2)} \dx u_1\dx %u_2=
(1+O(\varepsilon))\frac{2(\ell^*(u,t)-\sin\ell^* (u,t))}{\kappa^2_L(u)}.
\end{equation}

Integrating in \eqref{int:1} with respect to $u_1$ and $u_2$, and using Lemma~\ref{lemma:h(n)}, we obtain
\begin{multline*}
\lim_{n\to\infty}\EE (f_0(K_{(n)}))n^{-\frac 13}=(1+O(\varepsilon))
\lim_{n\to\infty} n^{-\frac 13}\frac{2}{A(K)^2}{n\choose 2}\\
\times\int_{S^1}\int_{0}^{h(n)}\left (1-\frac{A(u,t)}{A(K)}\right )^{n-2}
%\left(\frac{1}{\kappa_L(u)}-\frac{1}{\kappa_K(u)}+t\right)
k(u,t)\frac{\ell^*(u,t)-\sin\ell^* (u,t)}{\kappa^2_L(u)}\dx t\dx u.
\end{multline*}

Let $n_1$ be such that $0<h(n)<t_\varepsilon$ if $n>n_1$. Now assume, that $n>\max\left\{n_0,n_1\right\}$ and define 
\begin{equation}\label{eq:thetan}
\theta_n(u)= n^{-\frac 13}{n\choose 2}\int_{0}^{h(n)}\left (1-\frac{A(u,t)}{A(K)}\right )^{n-2}
%\left(\frac{1}{\kappa_L(u)}-\frac{1}{\kappa_K(u)}+t\right)
k(u,t) \frac{\ell^*(u,t)-\sin\ell^* (u,t)}{\kappa^2_L(u)}\dx t.
\end{equation}
Then
\begin{equation}\label{eq:Lebesgue}
\lim_{n\to\infty}\EE (f_0(K_{(n)}))n^{-\frac 13}\\=(1+O(\varepsilon))
\lim_{n\to\infty}\frac{2}{A(K)^2}\int_{S^1}\theta_n(u)\dx u.
\end{equation}
We recall from  \cite[(11), p. 2290]{BFRV09} that, for any $\beta\geq0, \omega>0$, and $\alpha>0$ it holds that
\begin{equation}\label{formula}
\int_{0}^{g(n)}t^{\beta}\left(1-\omega t^{\alpha}\right)^n\dx t \sim \frac{1}{\alpha\omega^{\frac{\beta+1}{\alpha}}}\Gamma\left(\frac{\beta+1}{\alpha}\right)n^{-\frac{\beta+1}{\alpha}},
\end{equation}
as $n\to\infty$, assuming that
$$\left(\frac{(\beta+\alpha+1)\ln n}{\alpha\omega n}\right)^{\frac{1}{\alpha}}<g(n)<\omega^{-\frac{1}{\alpha}}$$
for sufficiently large $n$.

In order to use Lebesgue's dominated convergence theorem for \eqref{eq:Lebesgue}, we need to show that the functions $\theta_n(u)$ are uniformly bounded on $S^1$. Clearly, there exists a $\gamma_3>0$ constant such that for all $0<t<t_\varepsilon$ and $u\in S^1$ we have $$\frac{\ell^*(u,t)-\sin\ell^* (u,t)}{\kappa^2_L(u)}<\gamma_3 t^{\frac 32},$$
and
$$k(u,t)=\frac{1}{\kappa_L(u)}-\frac{1}{\kappa_K(u)}+t<\gamma_4,$$
for a suitable $\gamma_4>0$ constant. From \eqref{formula} and \eqref{eq:thetan} with
\begin{align}
\alpha=\frac 32,\quad \beta=\frac 32,\quad \omega=\frac{\frac 23\sqrt{\frac{2}{1/\varrho-\kappa^*}}}{A(K)},
\end{align}
where $\kappa^*$ is the minimum of $\kappa_L(u)$ for $u\in S^1$, and $\varrho$ is the radius of the rolling circle of $K$ as in \eqref{eq:gordulo}, it follows that there exists $\gamma_5>0$ such that $\theta_n(u)<\gamma_5$ for all $u\in S^1$ and sufficiently large $n$. Thus, by Lebesgue's dominated convergence theorem
\begin{equation}
\lim_{n\to\infty}\EE (f_0(K_{(n)}))n^{-\frac 13}\\=(1+O(\varepsilon))
\frac{2}{A(K)^2}\int_{S^1}\lim_{n\to\infty}\theta_n(u)\dx u.
\end{equation}

Now assume that $0<t_\varepsilon$ is so small, that the following two conditions also hold for all $0<t<t_\varepsilon$ and $u\in S^1$
\begin{equation}\label{areaut}
(1-\varepsilon)\frac 43\sqrt{\frac{2}{\kappa_K(u)-\kappa_L(u)}}t^{\frac 32}<A(u,t)<(1+\varepsilon)\frac 43\sqrt{\frac{2}{\kappa_K(u)-\kappa_L(u)}}t^{\frac 32},
\end{equation}
and
\begin{align*}
%\label{ellut}
(1-\varepsilon)\frac 43 \left (\frac{2}{\kappa_K(u)-\kappa_L(u)}\right )^{\frac 32}t^{\frac 32}<\frac{\ell^3 (u,t)}{6}
%+O({\ell}^5(u,t))
&<(1+\varepsilon)\frac 43 \left (\frac{2}{\kappa_K(u)-\kappa_L(u)}\right )^{\frac 32}t^{\frac 32},
\end{align*}
as a result of Lemma~\ref{landAasym}. Using the Taylor series expansion of $\sin x$ around $0$, and the fact that $\lim_{t\to 0^+}\ell^*(u,t)/\ell(u,t)=\kappa_L(u)$, we obtain that for a sufficiently small $t_\varepsilon>0$ it holds that
\begin{align*}
%\label{kell2}
\frac{\ell^*(u,t)-\sin\ell^* (u,t)}{\kappa^2_L(u)}&=\frac{\left(\ell^*(u,t)\right)^3}{6\kappa^2_L(u)}+O\left((\ell^*(u,t))^5\right)\\
&=(1+O(\varepsilon))\kappa_L(u)
%\left(
\frac{\ell^3 (u,t)}{6}
%+O({\ell}^5(u,t))\right)
.
\end{align*}
Thus, we obtain
\begin{equation}\label{eq:lcsillag}
	\frac{\ell^*(u,t)-\sin\ell^* (u,t)}{\kappa^2_L(u)}=(1+O(\varepsilon))\kappa_L(u)\frac 43 \left (\frac{2}{\kappa_K(u)-\kappa_L(u)}\right )^{\frac 32}t^{\frac 32}.
\end{equation}
From \eqref{areaut} and \eqref{eq:lcsillag} it follows that
\begin{multline*}
\lim_{n\to\infty}\EE (f_0(K_{(n)}))n^{-\frac 13}\\=(1+O(\varepsilon))
\frac{2}{A(K)^2}\int_{S^1}\lim_{n\to\infty} n^{-\frac 13}{n\choose 2}\int_{0}^{h(n)}\left (1-\frac{\frac 43\sqrt{\frac{2}{\kappa_K(u)-\kappa_L(u)}}}{A(K)}t^{\frac 32}\right )^{n-2}\\\times
%\left(\frac{1}{\kappa_L(u)}-\frac{1}{\kappa_K(u)}+t\right)
k(u,t)\kappa_L(u)\frac 43 \left (\frac{2}{\kappa_K(u)-\kappa_L(u)}\right )^{\frac 32}t^{\frac 32}\dx t\dx u.
\end{multline*}
Therefore,
\begin{multline*}
\lim_{n\to\infty}\EE (f_0(K_{(n)}))n^{-\frac 13}=(1+O(\varepsilon))
\frac{4}{3A(K)^2}
%\cdot\frac 12\cdot\frac 43
\int_{S^1}\left (\frac{2}{\kappa_K(u)-\kappa_L(u)}\right )^{\frac 32}\kappa_L(u)\\\times\lim_{n\to\infty} n^{\frac 53}\int_{0}^{h(n)} t^{\frac 32}\left (1-\frac{\frac 43\sqrt{\frac{2}{\kappa_K(u)-\kappa_L(u)}}}{A(K)}t^{\frac 32}\right )^{n-2}k(u,t)
%\!\!\!\!\!\left(\frac{1}{\kappa_L(u)}-\frac{1}{\kappa_K(u)}+t\right)
\dx t\dx u.\notag
\end{multline*}
Now, the substitution
$\alpha=\beta=3/2$ and $\omega=(4\sqrt 2/(3A(K)))(\kappa_K(u)-\kappa_L(u))^{-1/2}$
%$$\alpha=\frac 32,\quad \beta=\frac 32,\quad \omega=\frac{\frac 43\sqrt{\frac{2}{\kappa_K(u)-\kappa_L(u)}}}{A(K)}$$
yields
\begin{align}
\lim_{n\to\infty}\EE (f_0(K_{(n)}))n^{-\frac 13}=(1+O(\varepsilon))
\frac{4}{3A(K)^2}
%\cdot\frac 12\cdot\frac 43
\int_{S^1}\left (\frac{2}{\kappa_K(u)-\kappa_L(u)}\right )^{\frac 32}\kappa_L(u)\notag\\\times\Bigg[\lim_{n\to\infty} n^{\frac 53}\int_{0}^{h(n)} t^{\beta}\left (1-\omega t^{\alpha}\right )^{n-2}\left(\frac{1}{\kappa_L(u)}-\frac{1}{\kappa_K(u)}\right)
\dx t\label{eq:elotag}\\
+\left.\lim_{n\to\infty} n^{\frac 53}\int_{0}^{h(n)} t^{\beta+1}\left (1-\omega t^{\alpha}\right )^{n-2}\dx t\right]\dx u.\label{eq:tag}
\end{align}
It follows from the asymptotic formula \eqref{formula} that the term \eqref{eq:tag} is $0$.
Applying \eqref{formula} to \eqref{eq:elotag}, we obtain
\begin{multline*}
\lim_{n\to\infty}\EE (f_0(K_{(n)}))n^{-\frac 13}=(1+O(\varepsilon))
\frac{2^{\frac 72}}{3A(K)^2}
%\cdot\frac 12\cdot\frac 43\cdot2^{3/2}
\int_{S^1}\left (\frac{1}{\kappa_K(u)-\kappa_L(u)}\right )^{\frac 32}\\\times \frac{\kappa_K(u)-\kappa_L(u)}{\kappa_K(u)} n^{\frac 53}\frac 23 \left(\frac{\frac 43\sqrt{\frac{2}{\kappa_K(u)-\kappa_L(u)}}}{A(K)}\right)^{-\frac 53}\Gamma\left(\frac 53\right)n^{-\frac 53}\dx u.
\end{multline*}

After simplification, we get
\begin{equation*}
\lim_{n\to\infty}\EE (f_0(K_{(n)}))n^{-\frac 13}=(1+O(\varepsilon))\sqrt[3]{\frac{2}{3A(K)}}\Gamma\left(\frac 53\right)\int_{S^1}\frac{\left(\kappa_K(u)-\kappa_L(u)\right)^{\frac 13}}{\kappa_K(u)}\dx u.
\end{equation*}
Since $\varepsilon>0$ was arbitrary, this completes the proof of Theorem~\ref{fotetel}. 

\begin{remark}
If $K\subset\R^2$ is a convex disc, then we denote integration on $\bd K$ with respect to the arc-length by $\int_{\bd K}\ldots \dx x$. It is well-known that if $\bd K$ is $C^2_+$, then for any measurable function $f(u)$ on  $S^1$ it holds that 
\begin{equation}\label{eq:arc-length}
\int_{S^1}f(u)\dx u=\int_{\bd K}f(u(K,x))\kappa_K(x)\dx x,
\end{equation}
see, for example, \cite[(2.5.30)]{Schneider}. With the help of \eqref{eq:arc-length}, the statement of Theorem~\ref{fotetel} can also be phrased slightly differently in the form
\begin{equation*}
\lim_{n\to\infty}\EE (f_0(K_{(n)}))n^{-\frac 13}=\sqrt[3]{\frac{2}{3A(K)}}\Gamma\left(\frac 53\right)\int_{\bd K}\left(\kappa_K(x)-\kappa_L(x)\right)^{\frac 13}\dx x.
\end{equation*}
\end{remark}

\section{The $K=L$ case} 

In this section we investigate the case when $K=L$. This is a direct generalization of Theorem 1.3 in \cite{FKV14}.

Now we turn to the proof of Theorem~\ref{LL}. Since the argument closely follows the proof of Theorem~\ref{fotetel} in Section~\ref{pf}, we only point out the major differences in the calculations.  First we note that the analogue of Lemma~\ref{capvertexexistsandunique} is true in the case $K=L$ as well. A general cap is of the form $C=\cl (L\setminus (L+p))$, and it clearly follows  that the vertex $x_0$ of the cap is the unique point in $\bd L$ where the unit outer normal is $-p/|p|$, while the height is $t=|p|$. We are going to use the notion $C(u,t)$, $A(u,t)$, etc. as before with the assumption that $K=L$. Next we need a variant of Lemma~\ref{landAasym}.

\begin{lemma}\label{lemma:ellAmod}
Let $L$ be a convex disc with $C^2_+$ boundary. Then
\begin{align}
	\lim_{t\to 0^+}\ell^*(u,t) =& \pi,\label{L-arcmod}\\
	\lim_{t\to 0^+}A(u,t)\cdot t^{-1}=& w(u) .\label{eq:Autmod}
	\end{align}
\end{lemma}

The proof of the Lemma~\ref{lemma:ellAmod} is simple, as the intersection points of $\bd L$ and $\bd L-tu$ tend to the points of tangency of the supporting lines $l_1(u)$ and $l_2(u)$ as $t\to 0+$.

We use the reparametrization $\Phi$ as introduced in (\ref{forjacobi}), and have that $$|J\Phi|=\frac{|u_1\times u_2|}{\kappa_L(u_1)\kappa_L(u_2)}t.$$ After the integral transformation we obtain
\begin{multline*}
\lim_{n\to\infty}\EE (f_0(L_{(n)}))\\=
\lim_{n\to\infty} \frac{1}{A(L)^2}{n\choose 2}\int_{S^1}\int_{0}^{t^*(u)}\int_{L^*(u,t)}\int_{L^*(u,t)}\left (1-\frac{A(u,t)}{A(L)}\right )^{n-2}\\
\times \frac{\left|u_1\times u_2\right|}{\kappa_L(u_1)\kappa_L(u_2)}t \dx u_1\dx u_2\dx t\dx u.
\end{multline*}

In the next step, as in Lemma~\ref{lemma:h(n)}, we split the domain of integration in $t$. From (\ref{eq:Autmod}) it follows that there is a universal constant $\hat c=\hat c(L)$ such that $A(u,t)>\hat ct$. We set $h(n)=c \ln n/n$ (where $c$ is a constant to be specified later), and obtain
\begin{multline*}
\lim_{n\to\infty}\EE (f_0(L_{(n)}))=(1+O(\varepsilon))
\lim_{n\to\infty} \frac{2}{A(L)^2}{n\choose 2}\int_{S^1}\int_{0}^{h(n)}\left (1-\frac{A(u,t)}{A(L)}\right )^{n-2}t \\ \times
\frac{\ell^*(u,t)-\sin\ell^* (u,t)}{\kappa^2_L(u)}\dx t\dx u,
\end{multline*}
where $\varepsilon>0$ is fixed as before, and we used (\ref{kell}).

Now, the inner integral is clearly uniformly bounded for $u\in S^1$, so we may use Lebesgue's dominated convergence theorem, hence
\begin{align*}
\lim_{n\to\infty}\EE (f_0(L_{(n)}))&=(1+O(\varepsilon))
 \frac{2}{A(L)^2} \int_{S^1} \lim_{n\to\infty}{n\choose 2}\int_{0}^{h(n)}\left (1-\frac{A(u,t)}{A(L)}\right )^{n-2}t \\ &\quad\quad\quad\times
\frac{\ell^*(u,t)-\sin\ell^* (u,t)}{\kappa^2_L(u)}\dx t\dx u\\
&=(1+O(\varepsilon))\pi\int_{S^1} \frac{1}{\kappa^2_L(u)\cdot w^2(u)} \dx u,
\end{align*}

where in the second step we used Lemma~\ref{lemma:ellAmod} and (\ref{formula}) (with a sufficiently large $c$). Since $\varepsilon>0$ was arbitrary, the proof of Theorem~\ref{LL} is now complete.

\section{Appendix}
In this section we calculate the Jacobian of the transformation $\Phi$ defined in \eqref{forjacobi}. We note that this Jacobian was already known to Santal\'o \cite{S46} in the special case when $L$ is circle of radius $r$. The calculation for $L=rS^1$ is also described in \cite{FKV14}*{Appendix A, pp. 916--917}. A $d$-dimensional generalization of the map $\Phi$ and its Jacobian was determined in \cite{F19}*{Lemma~2.2} for the case when $L=rB^d$. 

Let $\phi,\phi_1$ and $\phi_2$ be chosen such that the outer normals $u=(\cos\phi,\sin\phi)$, $u_i=(\cos\phi_i,\sin\phi_i)$, $i=1,2$. Then $\dx u\dx u_1\dx u_2=\dx \phi\dx \phi_1\dx \phi_2$. Furthermore, let $x(K,\phi)=x(K,u)$, $\kappa_K(\phi)=\kappa_K(x(K,\phi))$, and $x(L,\phi)=x(L,u)$, $\kappa_L(\phi)=\kappa_L(x(L,\phi))$.

Let $r_K:[0,2\pi)\to\bd K$ be a parametrization of $\bd K$ such that the outer unit normal $u(K, r_K(\phi))=(\cos\phi,\sin\phi)$. This parametrization is well-defined, and bijective due to the $C^2_+$ property of $\bd K$. Similarly, let $r_L:[0,2\pi)\to\bd L$ be a parametrization of $\bd L$ such that $u(L,r_L(\phi))=(\cos\phi, \sin \phi)$. Then $r_L$ is also well-defined and bijective, since $\bd L$ is also $C^2_+$. 

With the notation $x_i=(x_{i1}, x_{i2})$, $i=1,2$ for the Cartesian coordinates of $x_1$ and $x_2$, the Jacobian is 
$$J\Phi=\begin{vmatrix}
\frac{\partial x_{11}}{\partial \phi} & \frac{\partial x_{12}}{\partial \phi} & \frac{\partial x_{21}}{\partial \phi} & \frac{\partial x_{22}}{\partial \phi}\\[3pt] 
\frac{\partial x_{11}}{\partial t} & \frac{\partial x_{12}}{\partial t} & \frac{\partial x_{21}}{\partial t} & \frac{\partial x_{22}}{\partial t}\\[3pt] 
\frac{\partial x_{11}}{\partial \phi_1} & \frac{\partial x_{12}}{\partial \phi_1} &0 & 0\\[3pt] 
0 & 0 & \frac{\partial x_{21}}{\partial \phi_2} & \frac{\partial x_{22}}{\partial \phi_2}
\end{vmatrix},$$
where $$\frac{\partial x_{11}}{\partial t}=\frac{\partial x_{21}}{\partial t}=-\cos\phi, \quad\frac{\partial x_{12}}{\partial t}=\frac{\partial x_{22}}{\partial t}=-\sin\phi.$$
Note that  the special choice of $r_K(\phi)$ and $r_L(\phi)$ yields $\dx\phi=\kappa_K(\phi)\dx r_K$, $\dx\phi=\kappa_L(\phi)\dx r_L$ and $\dx\phi_i=\kappa_L(\phi_i)\dx r_L$ for $i=1,2$, thus

\begin{align*}
\frac{\partial x_{11}}{\partial \phi}&=\frac{\partial x_{21}}{\partial \phi}=\frac{-\sin\phi}{\kappa_K(\phi)}-\frac{-\sin\phi}{\kappa_L(\phi)}+t\sin\phi,\\
\frac{\partial x_{12}}{\partial \phi}&=\frac{\partial x_{22}}{\partial \phi}=\frac{\cos\phi}{\kappa_K(\phi)}-\frac{\cos\phi}{\kappa_L(\phi)}-t\cos\phi,\\
\frac{\partial x_{11}}{\partial \phi_1}&=\frac{-\sin\phi_1}{\kappa_L(\phi_1)},\quad\frac{\partial x_{12}}{\partial \phi_1}=\frac{\cos\phi_1}{\kappa_L(\phi_1)},\\
\frac{\partial x_{21}}{\partial \phi_2}&=\frac{-\sin\phi_2}{\kappa_L(\phi_2)},\quad\frac{\partial x_{22}}{\partial \phi_2}=\frac{\cos\phi_2}{\kappa_L(\phi_2)}.
\end{align*}

Now we can compute the determinant
\begin{align*}
J\Phi=&-\frac{\partial x_{21}}{\partial \phi_2} \begin{vmatrix}
\frac{\partial x_{11}}{\partial \phi} & \frac{\partial x_{12}}{\partial \phi} & \frac{\partial x_{22}}{\partial \phi}\\[3pt] 
\frac{\partial x_{11}}{\partial t} & \frac{\partial x_{12}}{\partial t} & \frac{\partial x_{22}}{\partial t}\\[3pt] 
\frac{\partial x_{11}}{\partial \phi_1} & \frac{\partial x_{12}}{\partial \phi_1}&0
\end{vmatrix}+\frac{\partial x_{22}}{\partial \phi_2}\begin{vmatrix}
\frac{\partial x_{11}}{\partial \phi}  & \frac{\partial x_{12}}{\partial \phi} & \frac{\partial x_{21}}{\partial \phi}\\[3pt] 
\frac{\partial x_{11}}{\partial t} & \frac{\partial x_{12}}{\partial t} & \frac{\partial x_{21}}{\partial t}\\[3pt] 
\frac{\partial x_{11}}{\partial \phi_1} & \frac{\partial x_{12}}{\partial \phi_1}&0
\end{vmatrix}\\
%=&-\frac{\partial x_{21}}{\partial \phi_2}\left[\frac{\partial x_{11}}{\partial \phi_1}\begin{vmatrix}
%\frac{\partial x_{12}}{\partial \phi} & \frac{\partial x_{22}}{\partial \phi}\\[3pt] 
%\frac{\partial x_{12}}{\partial t} & \frac{\partial x_{22}}{\partial t}
%\end{vmatrix}-\frac{\partial x_{12}}{\partial \phi_1}\begin{vmatrix}
%\frac{\partial x_{11}}{\partial \phi} & \frac{\partial x_{22}}{\partial \phi} \\[3pt] 
%\frac{\partial x_{11}}{\partial t} & \frac{\partial x_{22}}{\partial t}
%\end{vmatrix}\right]
%+\frac{\partial x_{22}}{\partial \phi_2}\left[\frac{\partial x_{11}}{\partial \phi_1}
%\begin{vmatrix}
%\frac{\partial x_{12}}{\partial \phi}  & \frac{\partial x_{21}}{\partial \phi}\\[3pt] 
%\frac{\partial x_{12}}{\partial t} & \frac{\partial x_{21}}{\partial t}
%\end{vmatrix}-\frac{\partial x_{12}}{\partial \phi_1}\begin{vmatrix}
%\frac{\partial x_{11}}{\partial \phi}  & \frac{\partial x_{21}}{\partial \phi} \\[3pt] 
%\frac{\partial x_{11}}{\partial t} & \frac{\partial x_{21}}{\partial t}\end{vmatrix}\right]\\
=&\left[-\frac{\partial x_{21}}{\partial \phi_2}\frac{\partial x_{22}}{\partial \phi}+\frac{\partial x_{22}}{\partial \phi_2}\frac{\partial x_{21}}{\partial \phi}\right]\left[\frac{\partial x_{11}}{\partial t}\frac{\partial x_{12}}{\partial \phi_1}-\frac{\partial x_{12}}{\partial t}\frac{\partial x_{11}}{\partial \phi_1}\right]-\\&-\left[-\frac{\partial x_{21}}{\partial \phi_2}\frac{\partial x_{22}}{\partial t}+\frac{\partial x_{22}}{\partial \phi_2}\frac{\partial x_{21}}{\partial t}\right]\left[\frac{\partial x_{11}}{\partial \phi}\frac{\partial x_{12}}{\partial \phi_1}-\frac{\partial x_{12}}{\partial \phi}\frac{\partial x_{11}}{\partial \phi_1}\right].
\end{align*}

Therefore, by substitution
\begin{multline*}
    J\Phi=\left[\frac{\sin\phi_2}{\kappa_L(\phi_2)}\frac{\partial x_{22}}{\partial \phi}+\frac{\cos\phi_2}{\kappa_L(\phi_2)}\frac{\partial x_{21}}{\partial \phi}\right]\left[-\cos\phi\frac{\cos\phi_1}{\kappa_L(\phi_1)}+\sin\phi\frac{-\sin\phi_1}{\kappa_L(\phi_1)}\right]-\\-\left[\frac{-\sin\phi_2}{\kappa_L(\phi_2)}(-\sin\phi)+\frac{\cos\phi_2}{\kappa_L(\phi_2)}(-\cos\phi)\right]\left[\frac{\partial x_{11}}{\partial \phi}\frac{\cos\phi_1}{\kappa_L(\phi_1)}+\frac{\partial x_{12}}{\partial \phi}\frac{\sin\phi_1}{\kappa_L(\phi_1)}\right]
    \end{multline*}

Thus,
$$\left|J\Phi\right|=\frac{\sin(\left|\phi_1-\phi_2\right|)}{\kappa_L(\phi_1)\kappa_L(\phi_2)}\cdot\left|\frac{1}{\kappa_K(\phi)}-\frac{1}{\kappa_L(\phi)}-t\right|.$$
We note that $\left|u_1\times u_2\right|$ equals the sine of the length of the unit circular arc between $u_1$ and $u_2$, that is, $\sin(\left|\phi_1-\phi_2\right|)=\left|u_1\times u_2\right|$. Furthermore, by the assumption on the curvatures $\kappa_K(\phi)>\kappa_L(\phi)$, we have
$$\left|J\Phi\right|=\frac{\left|u_1\times u_2\right|}{\kappa_L(u_1)\kappa_L(u_2)}\left(\frac{1}{\kappa_L(u)}-\frac{1}{\kappa_K(u)}+t\right),$$ which proves \eqref{jacobi}.

\end{document}